\date{\today}

\documentclass[10pt]{article}

\usepackage[T1]{fontenc}

\usepackage{amsmath}
\usepackage{amsthm}
\usepackage{amssymb}
\usepackage{url}
\usepackage{latexsym}
\usepackage{enumitem}
\usepackage{units} 

\usepackage{tikz}
\usepackage[square, sort, comma, numbers]{natbib}

\usepackage[small, it, singlelinecheck=false]{caption}

\usepackage[small, pagestyles]{titlesec}

\usepackage{color}
\definecolor{lightgray}{rgb}{0.8, 0.8, 0.8}
\definecolor{darkgray}{rgb}{0.65, 0.65, 0.65}

\usepackage[bookmarks]{hyperref}
\hypersetup{
        colorlinks=true,
        linkcolor=black,
        anchorcolor=black,
        citecolor=black,
        urlcolor=blue,
        pdfpagemode=UseThumbs,
        pdftitle={Three coloring via triangle counting},
        pdfsubject={Combinatorics},
        pdfauthor={Hamaker and Vatter},
}

\newcounter{todocounter}

\theoremstyle{plain}
\newtheorem{theorem}{Theorem}

\newtheorem{openproblem}[theorem]{Open Problem}

\theoremstyle{definition}

\makeatletter
\newfont{\footsc}{cmcsc10 at 8truept}
\newfont{\footbf}{cmbx10 at 8truept}
\newfont{\footrm}{cmr10 at 10truept}
\renewenvironment{abstract}%
                {
                  \begin{list}{}%
                     {\setlength{\rightmargin}{1in}%
                      \setlength{\leftmargin}{1in}}%
                   \item[]\ignorespaces\begin{small}}%
                 {\end{small}\unskip\end{list}}

\makeatletter
\let\start@align@nopar\start@align
\let\start@gather@nopar\start@gather
\let\start@multline@nopar\start@multline
\long\def\start@align{\par\start@align@nopar}
\long\def\start@gather{\par\start@gather@nopar}
\long\def\start@multline{\par\start@multline@nopar}
\makeatother

\newpagestyle{main}[\small]{
        \headrule
        \sethead[\usepage][][]
        {\sc Three Coloring via Triangle Counting}{}{\usepage}}

\addtocounter{footnote}{1}

\title{\sc Three Coloring via Triangle Counting}
\author{
	Zachary Hamaker%
	\footnote{Hamaker's research was partially supported by the NSF via award number 2054423.}
	\ and\ 
	Vincent Vatter%
	\footnote{Vatter's research was partially supported by the Simons Foundation via award number 636113.}
	\\
	\small Department of Mathematics\\[-0.5ex]
	\small University of Florida\\[-0.5ex]
	\small Gainesville, Florida USA
}

\titleformat{\section}
        {\large\sc}
        {\thesection.}{1em}{}

\begin{document}
\maketitle


\pagestyle{main}

\vspace{-0.3in}

\begin{abstract}
In the first partial result toward Steinberg's now-disproved three coloring conjecture, Abbott and Zhou used a counting argument to show that every planar graph without cycles of lengths~$4$ through~$11$ is $3$-colorable. Implicit in their proof is a fact about plane graphs: in any plane graph of minimum degree~$3$, if no two triangles share an edge, then triangles make up strictly less than~$\nicefrac{2}{3}$ of the faces. We show how this result, combined with Kostochka and Yancey's resolution of Ore's conjecture for~$k=4$, implies that every planar graph without cycles of lengths~$4$ through~$8$ is $3$-colorable.
\end{abstract}

%

%
%
%
%

In a 1975 letter, Steinberg asked if a planar graph without $4$- or $5$-cycles is necessarily \mbox{$3$-colorable}~\cite[Problem~9.1]{steinberg:the-state-of-th:}. There was little to no progress on Steinberg's conjecture until~1990. Surely some of this lack of progress was because Steinberg's conjecture is actually false, as established in~2017:

\begin{theorem}[Cohen-Addad, Hebdige, Kr\'{a}l', Li, and Salgado~\cite{CHKLS2016}]
\label{thm:3-color:steinberg:false}
There exists a planar graph without cycles of length $4$ or $5$ that is not $3$-colorable.
\end{theorem}

In~1990, Erd\H{o}s asked~\cite[Problem~9.2]{steinberg:the-state-of-th:} if there is an integer $k$ such every planar graph without cycles lengths $4$ through $k$ is $3$-colorable. The first answer to Erd\H{o}s's conjecture appeared only a year after he posed it.

\begin{theorem}[Abbott and Zhou~\cite{abbott:on-small-faces-:}]
\label{thm:3-color:lengths:4:11}
Every planar graph without cycles of lengths~$4$ through~$11$ is $3$-colorable.
\end{theorem}

Abbott and Zhou's proof was at its heart a counting argument.
A series of improvements to Theorem~\ref{thm:3-color:lengths:4:11} have been achieved, all using discharging rather than counting arguments. First, Borodin~\cite{borodin:to-the-paper-of:} proved that it suffices to forbid cycles of lengths $4$ through $10$. Then, Borodin~\cite{Borodin1996} and Sanders and Zhao~\cite{Sanders1995} proved independently that it suffices to forbid cycles of lengths $4$ through $9$. The current state-of-the-art is the following.

\begin{theorem}[Borodin, Glebov, Raspaud, and Salavatipour~\cite{borodin:planar-graphs-w:}]
\label{thm:3-color:lengths:4:7}
Every planar graph without cycles of lengths~$4$ through~$7$ is $3$-colorable.
\end{theorem}

Given that Theorem~\ref{thm:3-color:steinberg:false} shows that forbidding cycles of lengths $4$ and $5$ does not ensure a $3$-coloring, this leaves an open problem.

\begin{openproblem}
If a planar graph does not have cycles of lengths $4$, $5$, or $6$, is it necessarily $3$-colorable?
\end{openproblem}

Our goal in this note is to revisit Abbott and Zhou's proof of Theorem~\ref{thm:3-color:lengths:4:11} and show how combining their approach with a recent theorem of Kostochka and Yancey yields a result nearly as good as Theorem~\ref{thm:3-color:lengths:4:7} with very little effort. We begin by making explicit a result about plane graphs that is hidden in Abbott and Zhou's proof of Theorem~\ref{thm:3-color:lengths:4:11}:

\begin{theorem}
\label{thm:no:diamond:two-thirds:triangles}
If~$G$ is a connected plane graph of minimum degree $3$ in which no two triangles share an edge, then triangles make up strictly less than $\nicefrac{2}{3}$ of its faces.	
\end{theorem}

\begin{proof}
Let~$G$ be a connected plane graph with~$n$ vertices,~$e$ edges, and $f$ faces. Further let $n_3$ denote the number of degree $3$ vertices in~$G$, let $f_3$ denote the number of triangular faces of~$G$, and let $e_3$ denote the number of edges contained in a triangular face. Note that since no two triangles share an edge, $f_3=e_3/3$. By double counting edges, since the minimum degree of~$G$ is $3$, we have
\[
	2e
	=
	\sum_{v\in V(G)} \deg v
	\ge
	3n_3+4(n-n_3)
	=
	4n-n_3,
\]
so $n_3\ge 4n-2e$.

Now let $v$ be a vertex of degree $3$ in~$G$. Since no edge is contained in two triangles, at least one of the edges incident to $v$ must not be part of a triangle, and so contributes to~$e-e_3$. As this edge might be incident to two vertices of degree $3$, the most we can claim is that~$e-e_3\ge n_3/2$, or after rearranging, $e_3\le e-n_3/2$. Combining this with our inequality on~$n_3$, we have
\[
	f_3
	=
	\frac{e_3}{3}
	\le
	\frac{e-n_3/2}{3}
	\le
	\frac{2e-2n}{3}
	=
	\frac{2f-4}{3},
\]
where the final equality follows by Euler's formula, $f+n=e+2$. This proves the result.
\end{proof}

Theorem~\ref{thm:no:diamond:two-thirds:triangles} quickly leads to a proof of Theorem~\ref{thm:3-color:lengths:4:11}:

\begin{proof}[Proof of Theorem~\ref{thm:3-color:lengths:4:11}]
Let~$G$ be a plane graph with~$n$ vertices,~$e$ edges, and $f$ faces, and without cycles of lengths $4$ through $11$. We prove the result by induction on~$n$, the base case~$n=0$ holding trivially. If~$G$ has a vertex $v$ of degree at most~$2$, then $G-v$ is $3$-colorable by induction, and we may extend such a coloring to $3$-color~$G$. Thus we may assume that the minimum degree of~$G$ is $3$. Similarly, we may assume that~$G$ is connected.

Let $f_3$ denote the number of triangles in~$G$.
No two triangles of~$G$ may share an edge because~$G$ does not contain any $4$-cycles, so $f_3<2f/3$ by Theorem~\ref{thm:no:diamond:two-thirds:triangles}. As every edge lies on two faces and every non-triangular face of~$G$ has at least $12$ edges, the number of non-triangular faces of~$G$ satisfies $f-f_3\le (2e-3f_3)/12$. Thus we have
\begin{equation}
\label{eqn:non-triangle:bound}
\tag{$\dagger$}
	f
	\le
	f_3+
	\frac{2e-3f_3}{12}
	=
	\frac{e}{6}+\frac{3f_3}{4}
	<
	\frac{e}{6}+\frac{f}{2},
\end{equation}
so $f<e/3$. By Euler's formula we have $e=n+f-2$, so
\begin{equation}
\label{eqn:euler:ending}
\tag{$\ddagger$}
	e
	=
	n+f-2
	<
	n+\frac{e}{3}-2,
\end{equation}
and thus $e<3n/2-3$. This proves that~$G$ has average degree less than $3$. Therefore, it must contain a vertex of degree at most $2$, and so can be $3$-colored by induction.
\end{proof}

If cycles of length $11$ are allowed, then the inequality in \eqref{eqn:non-triangle:bound} must be changed to
\[
	f
	\le
	f_3+\frac{2e-3f_3}{11}
	=
	\frac{2e}{11}+\frac{8f_3}{11}
	<
	\frac{2e}{11}+\frac{16f}{33}.
\]
This implies that $f<6e/17$, so~\eqref{eqn:euler:ending} becomes
\[
	e
	=
	n+f-2
	<
	n+\frac{6e}{17}-2,
\]
which implies that $e<17n/11-34/11$. This is not enough to guarantee a vertex of degree at most~$2$, and so the argument used by Abbott and Zhou cannot be used to prove a result stronger than Theorem~\ref{thm:3-color:lengths:4:11}.

There is, however, a different way to use Theorem~\ref{thm:no:diamond:two-thirds:triangles} to prove a result about $3$-coloring planar graphs without cycles. A graph is \emph{$k$-critical} if it has chromatic number $k$, but all of its induced subgraphs have chromatic number strictly less than $k$. Kostochka and Yancey~\cite{kostochka:ores-conjecture:} recently nearly resolved Ore's conjecture on the minimum number of edges in a $k$-critical graph. They also gave~\cite{kostochka:ores-conjecture:k4} a short and self-contained proof in the case $k=4$, where the result reduces to the following.

\begin{theorem}[Kostochka and Yancey~\cite{kostochka:ores-conjecture:k4,kostochka:ores-conjecture:}]
\label{thm:ore:k=4}
If~$G$ is a $4$-critical graph with~$n$~vertices and~$e$~edges, then
\[
	e\ge\frac{5n-2}{3}.
\]
\end{theorem}

Kostochka and Yancey~\cite{kostochka:ores-conjecture:k4} showed how Theorem~\ref{thm:ore:k=4} leads to a very short proof of Gr\"{o}tsch's celebrated three color theorem (every triangle-free planar graph is $3$-colorable). Borodin, Kostochka, Lidick\'{y}, and Yancey~\cite{borodin:short-proofs-of:} later showed how Theorem~\ref{thm:ore:k=4} can also be used to give a short proof of Gr\"unbaum's three color theorem (every planar graph with at most three triangles is $3$-colorable). Below, we use Theorem~\ref{thm:ore:k=4} together with the bound on triangles given by Theorem~\ref{thm:no:diamond:two-thirds:triangles} to derive a result nearly as good as Theorem~\ref{thm:3-color:lengths:4:7}.

\begin{theorem}
Every planar graph without cycles of lengths $4$ through $8$ is $3$-colorable.
\end{theorem}

\begin{proof}
Suppose that the result is not true and take~$G$ to be a plane graph of minimal order, say~$n$, that is not $3$-colorable despite having no cycles of lengths $4$ through $8$.
Let~$e$ denote the number of edges of~$G$ and $f$ denote the number of faces.
As it is a minimal counterexample,~$G$ must be $4$-critical, so we have $e\ge 5n/3-2/3$ by Theorem~\ref{thm:ore:k=4}. Let $f_3$ denote the number of triangles in~$G$; again we have $f_3<2f/3$ by Theorem~\ref{thm:no:diamond:two-thirds:triangles}. As the shortest non-triangular faces of~$G$ have length $9$, the inequality \eqref{eqn:non-triangle:bound} in our proof of Theorem~\ref{thm:3-color:lengths:4:11} becomes
\[
	f
	\le
	f_3+\frac{2e-3f_3}{9}
	=
	\frac{2e}{9}+\frac{2f_3}{3}
	<
	\frac{2e}{9}+\frac{4f}{9}.
\]
This implies that $f<2e/5$, so by applying Euler's formula, the inequality~\eqref{eqn:euler:ending} becomes
\[
	e
	=
	n+f-2
	<
	n+\frac{2e}{5}-2.
\]
However, this shows that $e<5n/3-10/3$, which contradicts the fact that $e\ge 5n/3-2/3$.
\end{proof}

\section*{Acknowledgements}
We thank Zachary Hunter for discovering a mistake in a previous version of this paper.

%
%
%
%
%
%
%


\begin{thebibliography}{10}

\bibitem{abbott:on-small-faces-:}
{\sc Abbott, H.~L., and Zhou, B.}
\newblock On small faces in {$4$}-critical planar graphs.
\newblock {\em Ars Combin. 32\/} (1991), 203--207.

\bibitem{Borodin1996}
{\sc Borodin, O.~V.}
\newblock Structural properties of plane graphs without adjacent triangles and
  an application to {$3$}-colorings.
\newblock {\em J. Graph Theory 21}, 2 (1996), 183--186.

\bibitem{borodin:to-the-paper-of:}
{\sc Borodin, O.~V.}
\newblock To the paper of {H.L. Abbott and B. Zhou} on $4$-critical planar
  graphs.
\newblock {\em Ars Combin. 43\/} (1996), 191--192.

\bibitem{borodin:planar-graphs-w:}
{\sc Borodin, O.~V., Glebov, A.~N., Raspaud, A., and Salavatipour, M.~R.}
\newblock Planar graphs without cycles of length from $4$ to $7$ are
  $3$-colorable.
\newblock {\em J. Combin. Theory Ser. B 93}, 2 (2005), 303--311.

\bibitem{borodin:short-proofs-of:}
{\sc Borodin, O.~V., Kostochka, A.~V., Lidick\'{y}, B., and Yancey, M.~P.}
\newblock Short proofs of coloring theorems on planar graphs.
\newblock {\em European J. Combin. 36\/} (2014), 314--321.

\bibitem{CHKLS2016}
{\sc Cohen-Addad, V., Hebdige, M., Kr\'{a}l', D., Li, Z., and Salgado, E.}
\newblock Steinberg's conjecture is false.
\newblock {\em J. Combin. Theory Ser. B 122\/} (2017), 452--456.

\bibitem{kostochka:ores-conjecture:k4}
{\sc Kostochka, A.~V., and Yancey, M.~P.}
\newblock Ore's conjecture for {$k=4$} and {G}r\"{o}tzsch's theorem.
\newblock {\em Combinatorica 34}, 3 (2014), 323--329.

\bibitem{kostochka:ores-conjecture:}
{\sc Kostochka, A.~V., and Yancey, M.~P.}
\newblock Ore's conjecture on color-critical graphs is almost true.
\newblock {\em J. Combin. Theory Ser. B 109\/} (2014), 73--101.

\bibitem{Sanders1995}
{\sc Sanders, D.~P., and Zhao, Y.}
\newblock A note on the three color problem.
\newblock {\em Graphs Combin. 11}, 1 (1995), 91--94.

\bibitem{steinberg:the-state-of-th:}
{\sc Steinberg, R.}
\newblock The state of the three color problem.
\newblock In {\em Quo Vadis, Graph Theory?}, J.~Gimbel, J.~W. Kennedy, and
  L.~V. Quintas, Eds., vol.~55 of {\em Ann. Discrete Math.} North-Holland,
  Amsterdam, The Netherlands, 1993, pp.~211--248.

\end{thebibliography}

\end{document}